\documentclass[a4paper]{amsart}
\usepackage[utf8]{inputenc}
\usepackage[english]{babel}
\usepackage{amsfonts}
\usepackage{amssymb}
\usepackage{amsthm}
\usepackage{mathtools}
\usepackage[left=3cm, right=3cm]{geometry}

\usepackage{tikz}
\usepackage{pgfplots}
\pgfplotsset{compat=1.6}
\usetikzlibrary{intersections,patterns,pgfplots.fillbetween}
\pgfdeclarelayer{bg}
\pgfsetlayers{bg,main}
\makeatletter 
\pgfdeclarepatternformonly[\LineSpace]{my north east lines}{\pgfqpoint{-1pt}{-1pt}}{\pgfqpoint{\LineSpace}{\LineSpace}}{\pgfqpoint{\LineSpace}{\LineSpace}}%
{
	\pgfsetcolor{\tikz@pattern@color} 
	\pgfsetlinewidth{0.1pt}
	\pgfpathmoveto{\pgfqpoint{0pt}{0pt}}
	\pgfpathlineto{\pgfqpoint{\LineSpace + 0.1pt}{\LineSpace + 0.1pt}}
	\pgfusepath{stroke}
}
\makeatother 
\newdimen\LineSpace
\tikzset{
	line space/.code={\LineSpace=#1},
	line space=10pt
}

\usepackage{mathtools}
\mathtoolsset{showonlyrefs}

\usepackage{hyperref}
\usepackage{xcolor}
\hypersetup{
    colorlinks,
    linkcolor={blue!90!black},
    citecolor={green!50!black},
    urlcolor={blue!90!black}
}

\makeatletter
\def\blfootnote{\gdef\@thefnmark{}\@footnotetext}
\makeatother

\newcommand{\R}{\mathbb R}
\newcommand{\C}{\mathbb C}

\newcommand{\mc}{\mathcal}

\renewcommand{\phi}{\varphi}
\newcommand{\n}{\nabla}
\newcommand{\pd}{\partial}
\newcommand{\tr}{{\rm tr}}
\newcommand{\scal}{\widehat{\mathrm{sc}}}

\renewcommand{\ge}{\geqslant}
\renewcommand{\le}{\leqslant}
\newcommand{\mf}{\mathfrak}
\newcommand{\Ad}{\mathrm{Ad}}
\newcommand{\ad}{\mathrm{ad}}
\newcommand{\Sym}{\mathrm{Sym}}
\newcommand{\End}{\mathrm{End}}
\newcommand{\la}{\langle}
\newcommand{\ra}{\rangle}

\renewcommand{\bar}[1]{\mkern 0.5mu\overline{\mkern-0.5mu#1\mkern-0.5mu}\mkern 0.5mu}

\theoremstyle{plain}
\newtheorem{theorem}{Theorem}[section]
\newtheorem*{theorem*}{Theorem}
\newtheorem{lemma}[theorem]{Lemma}
\newtheorem{problem}[theorem]{Problem}
\newtheorem{corollary}[theorem]{Corollary}
\newtheorem{proposition}[theorem]{Proposition}

\theoremstyle{definition}
\newtheorem{remark}[theorem]{Remark}

\newtheorem{definition}[theorem]{Definition}
\newtheorem{example}[theorem]{Example}

\begin{document}
\title[Lie-Algebraic Curvature Conditions preserved by HCF]{Lie-Algebraic Curvature Conditions \\preserved by the Hermitian Curvature Flow}
\date{}
\author{Yury Ustinovskiy}
\address{Fine Hall, Princeton University, Princeton, NJ, 08540}
\email{yuryu@math.princeton.edu}

\begin{abstract}
	The purpose of this paper is to prove that the Hermitian Curvature Flow (HCF) on an Hermitian manifold $(M,g,J)$ preserves many natural curvature positivity conditions. Following Wilking~\cite{wi-13}, for an $\Ad\,{GL(T^{1,0}M)}$-invariant subset $S\subset \End(T^{1,0}M)$ and a \emph{nice} function $F\colon \End(T^{1,0}M)\to\R$ we construct a convex set of curvature operators $C(S,F)$, which is invariant under the HCF. Varying $S$ and $F$, we prove that the HCF preserves Griffiths positivity, Dual-Nakano positivity, positivity of holomorphic orthogonal bisectional curvature, lower bounds on the second scalar curvature. As an application, we prove that periodic solutions to the HCF can exist only on manifolds $M$ with the trivial canonical bundle on the universal cover $\widetilde{M}$.
\end{abstract}
\maketitle
\section*{Introduction}\label{sec:intro}
In the last decades the Ricci flow has been successfully used in many classification and uniformization problems in Riemannian and K\"ahler geometry. The direction started in the 1980's with Hamilton's original papers on the classification of three/four-dimensional manifolds admitting metrics with positive Ricci curvature/positive curvature operator~\cite{ha-82,ha-86}. The main idea in Hamilton's approach is to control positivity of the curvature tensor along the Ricci flow, using a form of the parabolic maximum principle for tensors. Following this route one can prove the pinching of the curvature tensor towards a constant curvature tensor. In 2000's this program was used by Brendle and Schoen~\cite{br-sc-09,br-sc-08} to classify manifolds with (weakly) 1/4-pinched sectional curvature and by B\"ohm and Wilking~\cite{bo-wi-08} to describe manifolds admitting positive curvature operator in all dimensions. In the K\"ahler setting, Chen, Song and Tian~\cite{ch-so-ti-09} gave a Ricci flow-based proof of the Frankel conjecture~\cite{fr-61}. This conjecture states that the existence of a K\"ahler metric of positive holomorphic bisectional curvature on a complex manifold $M$ implies that $M$ is biholomorphic to a complex projective space; it was solved by Siu and Yau~\cite{si-ya-80} by studying the space of harmonic maps $S^2\to M$. The authors of~\cite{ch-so-ti-09} proved that any K\"ahler metrics of positive holomorphic bisectional curvature under the Ricci flow pinches towards the metric of constant curvature.

Unlike the K\"ahler situation, in a general Hermitian setting there are very few efficient analytic tools. For this reason it is interesting to extend the Ricci flow onto Hermitian manifolds. Unfortunately, the Ricci flow itself is not well-suited for the category of Hermitian manifolds, since, on a general Hermitian manifold $(M,g,J)$, the Ricci tensor $Ric(g)$ is not necessarily $J$-invariant. Motivated by this problem, Streets ant Tian~\cite{st-ti-11} introduced a family of Hermitian Curvature Flows, generalizing the Ricci flow:
\begin{equation}\label{eq:HCF_general}
\pd_t g=-S^{(2)}+Q(T),
\end{equation}
where $S^{(2)}$ it the second Chern-Ricci curvature (see formula~\eqref{eq:Ricci_contractions} below) and $Q(T)$ is an arbitrary type-(1,1) quadratic term in torsion of $g$. If $(M,g,J)$ is K\"ahler, then $S^{(2)}=Ric(g)$, $T=0$, and all the Hermitian Curvature Flows coincide with the K\"ahler-Ricci flow. Recently, in a series of papers~\cite{st-ti-10,st-ti-12,st-16,st-16-2} Streets and Tian investigated a specific member of family~\eqref{eq:HCF_general} with $Q(T)_{i\bar j}=g^{m\bar n}g_{p\bar s}T^p_{i m}T_{\bar j\bar n}^{\bar s}$, which they call the~\emph{pluriclosed flow}. This flow preserves the class of pluriclosed metrics (i.e, $g$, s.t., $\pd\bar{\pd}\omega_g=0$, where $\omega_g=g(J\cdot,\cdot)$), and in many geometric situations satisfies global existence and convergence results.

In this paper we study \emph{another} member of~\eqref{eq:HCF_general} with $Q(T)_{i\bar j}=-1/2g^{m\bar n}g^{p\bar s}T_{mp\bar j}T_{\bar n\bar s i}$. This specific flow was first introduced in~\cite{us-16}, where we proved that it preserves Griffiths positivity (non-negativity) of $(TM,g)$. We refer to this flow as the HCF (Hermitian Curvature Flow). The choice of the quadratic term $Q(T)$ is motivated by a very special evolution equation for the Chern curvature $\Omega$ under~\eqref{eq:HCF_general}. In~\cite{us-17} we computed the HCF for the induced metrics on all complex homogeneous manifolds. In the present paper, we reinterpret this equation, by setting a space of algebraic curvature tensors and introducing natural operations on this space. It allows us to find a much clearer expression for $\pd_t \Omega$ (Proposition~\ref{prop:omega_evolution_clear}).

Our principle goal in this paper is to prove that the HCF preserves many natural curvature positivity conditions, besides Griffiths positivity. Following Wilking~\cite{wi-13}, for every $\Ad\,GL(T^{1,0}M)$-invariant $S\subset \End(T^{1,0}M)$ and any \emph{nice} function (see Definition~\ref{def:nice}) $F\colon \End(T^{1,0}M)\to \R$, we define a closed convex set of curvature-type tensors $C(S,F)$. We prove a modification of Hamilton's maximum principle~\cite{ha-86} (Section~\ref{sec:max_principle}), and adopt arguments of~\cite{wi-13} to prove that sets $C(S,F)$ are preserved by the HCF. Algebraic structure of the evolution equation for $\Omega$ plays the crucial role in this proof.
\begin{theorem}
	For any $\Ad\,GL(T^{1,0}M)$-invariant subset $S\subset \End(T^{1,0}M)$ and any nice function $F\colon \End(T^{1,0}M)\to\R$, the curvature condition $C(S,F)$ is preserved by the HCF.
\end{theorem}

We also prove a \emph{strong} version of the above theorem. Namely, we characterize the set of points, where $\Omega$, corresponding to the evolved metric, hits the boundary of $C(S,F)$ (Theorem~\ref{thm:br-sc}). Choosing various $S$ and $F$, we prove that the HCF preserves Griffiths positivity, Dual-Nakano positivity, positivity of holomorphic orthogonal bisectional curvature, lower bounds on the second scalar curvature (Section~\ref{sec:examples}). The latter has important consequences for the understanding of possible stationary and periodic solutions to the HCF. In particular, we prove:
\begin{theorem}
	If a compact complex manifold $M$ admits an HCF-periodic Hermitian metric, then the pull-back of the canonical bundle to the universal cover of $M$ is holomorphically trivial.
\end{theorem}

We conjecture that, as in the Ricci flow case, the HCF admits only stationary periodic solutions.

There exist other Hermitian generalizations of the Ricci flow. Among them we mention the Chern-Ricci flow, introduced ans investigated by Gill~\cite{gi-11}. Its long-time existence and convergence properties were also studied by Tosatti and Weinkove \cite{to-we-13,to-we-15}. Under this flow, the metric form $\omega_g$ is evolved in the direction of the Chern-Ricci form $\rho$, which represents (the multiple of) the first Chern class of $M$.

The rest of the paper is organized as follows. In Section~\ref{sec:background} we fix basic notations of Hermitian geometry, set up the space of algebraic Chern curvature tensors, define natural operations on this space and introduce the HCF. We also recall some of the computations of~\cite{us-16}. Next, in Section~\ref{sec:max_principle} we formulate and reprove Hamilton's maximum principle for tensors in a slightly more general context. With this generalization Hamilton's maximum principle becomes applicable to the evolution equation for $\Omega$ under the HCF. In Section~\ref{sec:main} define the convex sets $C(S,F)$ of curvature-type tensors and prove that they are preserved under the HCF. This section essentially follows~\cite{wi-13}. We provide several examples of the HCF-preserved curvature conditions in Section~\ref{sec:examples}. Finally, we discuss some applications and further questions in Section~\ref{sec:applications}.

\section*{Acknowledgements}
I would like to thank my advisor Gang Tian for introducing me a circle of problems concerning Hermitian Curvature Flows, and for constant support in my research. I also wish to thank Jeffrey Streets for valuable discussions.

\section{Background}\label{sec:background}

This section consists of three parts. First, we provide some background on Hermitian geometry and set up notations, introducing basic objects, such as Chern connection, its torsion and the Chern curvature tensors. In the second part, we define a vector space of algebraic curvature tensors, and present some natural algebraic operations on it.
Lastly, we define the Hermitian Curvature Flow (HCF). Using the aforementioned operations we write down the evolution equation for the Chern curvature under the HCF.

\subsection{Hermitian geometry}
For a complex vector space $V$, we denote by $\bar V$ the underlying real vector space with a conjugate complex structure. We denote by $\Sym^{1,1}(V)$ the subspace of $V\otimes \bar V$ spanned over $\R$ by all the elements of the form $v\otimes\bar v$, $v\in V$.

Let $(M,J)$ be a compact complex manifold with the operator of almost complex structure $J\colon TM\to TM$. Denote by $TM\otimes \C=T^{1,0}M\oplus T^{0,1}M$ the decomposition the complexified tangent bundle into $\pm \sqrt{-1}$-eigenspaces of $J$. Any $J$-invariant Riemannian metric $g$ defines an Hermitian metric on $T^{1,0}M$. We denote this Hermitian metric by the same symbol.

\emph{Chern connection} on $TM$ is the connection $\n$ characterized by the following properties:
\begin{enumerate}
	\item $\n g=0$,
	\item $\n J=0$,
	\item $T(X,JY)=T(JX,Y)$ for $X,Y\in TM$,
\end{enumerate}
where $T(X,Y):=\n_X Y-\n_Y X-[X,Y]$ is the torsion tensor. There exists a unique connection satisfying these properties. After extending $\n$ to a $\C$-linear connection on $TM\otimes\C$, we can substitute property (3) with either of the following properties
\begin{enumerate}
	\item[($3_1$)] $T(\xi,\bar{\eta})=0$ for $\xi,\eta\in T^{1,0}M$,
	\item[($3_2$)] $\n_{\bar \xi}\eta=i_\xi\bar\pd\eta$ for $\xi,\eta\in T^{1,0}M$,
\end{enumerate}
where $\bar\pd\colon\Gamma(T^{1,0}M)\to\Gamma(\Lambda^{0,1}\otimes T^{1,0}M)$ is the operator of the holomorphic structure.

The \emph{Chern curvature} $\Omega$ is the curvature of the Chern connection, namely,
\begin{equation}\label{eq:chern_curvature}
\Omega(X,Y,Z,W):=g((\n_X\n_Y-\n_Y\n_X-\n_{[X,Y]})Z,W).
\end{equation}
Extend $\Omega$ to a $\C$-linear tensor. Chern curvature tensor has many symmetries: it is antisymmetric and $J$-invariant in its first and second pairs of arguments, i.e.,
\[
\Omega(X,Y,Z,W)=\Omega(JX,JY,Z,W)=\Omega(X,Y,JZ,JW),
\]
\[
\Omega(X,Y,Z,W)=-\Omega(Y,X,Z,W)=-\Omega(X,Y,W,X).
\]
These symmetries imply that $\Omega$ is completely determined by the components
\[
\Omega_{i\bar j k\bar l}:=\Omega(\pd/\pd z_i,\pd/\pd \bar{z_j}, \pd/\pd z_k,\pd/\pd \bar{z_l}),
\]
where $\{\pd/\pd z_i\}$ is a local frame. If we want to specify explicitly that $\Omega$ corresponds to a metric $g$, we write $\Omega=\Omega^g$. We also denote the components of the torsion tensor (assuming the Einstein summation convention) as
\[
T(\pd/\pd z_i,\pd/\pd z_j):=T_{ij}^k\pd/\pd z_k,
\]
\[
T_{ij\bar l}:=T_{ij}^kg_{k\bar l}.
\]
If $g_{i\bar j}$ is a coordinate expression for the metric, then the components of the torsion and the curvature tensors are given by
\[
T_{ij}^k=g^{k\bar l}(\pd_i g_{j\bar l}-\pd_jg_{i\bar l}),
\]
\[
\Omega_{i\bar j k\bar l}=-\pd_{i}\pd_{\bar j}g_{k\bar l}+g^{p\bar s}\pd_{\bar j}g_{p\bar l}\pd_{i}g_{k\bar s}.
\]
Unlike the Riemannian case, the Chern curvature does not satisfy the classical Bianchi identities, since the Chern connection has torsion. However, in this case, slightly modified identities, involving torsion still hold~\cite[Ch.\,III, Thm.\,5.3]{ko-no-63}.

\begin{proposition}[Bianchi identities for the Chern curvature and torsion]
	\begin{equation}
	\begin{split}
	\Omega_{i \bar{j} k \bar{l}}=\Omega_{k \bar{j} i \bar{l}} + \n_{\bar{j}} T_{ki\bar l},\quad &
	\Omega_{i \bar{j} k \bar{l}}=\Omega_{i \bar{l} k \bar{j}} + \n_{i} T_{\bar{l}\bar{j} k},\\
	\n_m \Omega_{i \bar{j} k \bar{l}} = \n_i\Omega_{m \bar{j} k \bar{l}} + T^p_{i m}\Omega_{p \bar{j} k \bar{l}},\quad &
	\n_{\bar n} \Omega_{i \bar{j} k \bar{l}} = \n_{\bar j}\Omega_{i \bar{n} k \bar{l}} + T^{\bar s}_{\bar j \bar n}\Omega_{i \bar{s} k \bar{l}},\\
	\n_i T_{jk}^l+\n_k T_{ij}^l+\n_j T_{ki}^l=&
	T_{ij}^pT_{kp}^l+T_{jk}^pT_{ip}^l+T_{ki}^pT_{jp}^l.
	\end{split}
	\end{equation}
\end{proposition}

Presence of the torsion terms in the Bianchi identities for $\Omega$ indicates that the four Ricci contractions of the Chern curvature tensor differ from each other (see~\cite{li-ya-15} for the explicit description of the differences between these contractions):
\begin{equation}\label{eq:Ricci_contractions}
\begin{split}
S^{(1)}_{i\bar j}:=\Omega_{i\bar jm\bar n}g^{m\bar n},\quad 
S^{(2)}_{i\bar j}:=\Omega_{m\bar ni\bar j}g^{m\bar n},\\
S^{(3)}_{i\bar j}:=\Omega_{n\bar ji\bar m}g^{m\bar n},\quad 
S^{(4)}_{i\bar j}:=\Omega_{i\bar nm\bar j}g^{m\bar n}.\\
\end{split}
\end{equation}
We call these contractions the Chern-Ricci tensors. All tensors will play certain role below. Symmetries of $\Omega$ imply that the first and the second Chern-Ricci tensors define Hermitian products on $T^{1,0}M$. In general this is not the case for $S^{(3)}$ and $S^{(4)}$, since $S^{(3)}_{i\bar j}\neq\bar{S^{(3)}_{j\bar i}}=S^{(4)}_{i\bar j}$.

There are also two scalar contractions of $\Omega$: the scalar curvature $\mathrm{sc}=g^{i\bar j}g^{k\bar l}\Omega_{i\bar j k\bar l}=\tr_g S^{(1)}=\tr_g S^{(2)}$ and a quantity which will be referred to as the~\emph{second scalar curvature} and will play important role below: 
\begin{equation}\label{eq:scal_hat}
\scal=g^{i\bar l}g^{k\bar j}\Omega_{i\bar j k\bar l}=\tr_g S^{(3)}=\tr_g S^{(4)}.
\end{equation}

By rising the last two indices of the Chern curvature tensor, we can interpret $\Omega$ as a section of $\End(T^{1,0}M)\otimes\bar{\End(T^{1,0}M)}$:
\[
 		\Omega_{i\bar j}^{\ \ \bar l k}(e_k\otimes\epsilon^i)\otimes\bar{(e_l\otimes\epsilon^j)},
 	\quad \Omega_{i\bar j}^{\ \ \bar l k}= \Omega_{i\bar j m\bar n}g^{m\bar l}g^{k\bar n},
\]
where $\{e_i\}$ is a local frame of $T^{1,0}M$ and $\{\epsilon^i\}$ is the dual frame.
Symmetries of $\Omega$ imply that this form is Hermitian, i.e., lies in $\Sym^{1,1}(\End(T^{1,0}M))$.
\begin{remark}\label{rm:nakano}
	If $\Omega\in\Sym^{1,1}(\End(T^{1,0}M))$ is positive definite (resp.\,semidefinite), then the tangent bundle $T^{1,0}M$ is said to be \emph{dual-Nakano positive} (resp.\,non-negative),~\cite{de}. Dual-Nakano non-negative metrics exist on all complex homogeneous spaces, see Example~\ref{ex:homogeneous} below.
\end{remark}

\subsection{Algebra of the space of curvature tensors}\label{subsec:algebra}
Let $(V,g)$ be a complex vector space equipped with an Hermitian inner product. We extend $g$ to all associated tensor powers of $V$ and $\bar V$. Denote by $\mf g=\End(V)$ the endomorphism Lie algebra of $V$. Let $\la\cdot,\cdot\ra_{\tr}\colon\mf g\otimes\mf g\to\mf \C$ be the trace pairing
\[
\la u,v\ra_\tr:=\tr(uv).
\]

\begin{definition}
	The space of~\emph{algebraic curvature tensors} on $V$ is the vector space $\Sym^{1,1}(\mf g)$.
\end{definition}
Pairing $\la\cdot,\cdot\ra_\tr$ extends to a bilinear form on $\Sym^{1,1}(\mf g)$ in an obvious way:
\[
\la v\otimes \bar v, u\otimes\bar u\ra_\tr:=|\tr(uv)|^2.
\]
Clearly, $\Omega\in\Sym^{1,1}(\mf g)$ is positive (resp.\,non-negative) if and only if $\la\Omega,u\otimes\bar u\ra_\tr>0$ (resp.\,$\ge 0$) for any nonzero $u\in \mf g$.
\begin{remark}
	For $V=T^{1,0}M$ the space $\Sym^{1,1}(\mf g)$ models the space of Chern curvature tensors. Unlike the Riemannian/K\"ahler setting, where one is interested only in the part of $\Sym^{1,1}(\mf g)$, satisfying the algebraic Bianchi identity, we consider the whole space $\Sym^{1,1}(\mf g)$, since the Chern curvature has less symmetries.	
\end{remark}

There is a natural $\R$-linear adjoint action of $\mf g$ on $\Sym^{1,1}(\mf g)$:
\[
\ad_v(u\otimes \bar u)=[v,u]\otimes\bar u+u\otimes\bar{[v,u]},\quad v,u\in\mf g.
\]
For $\Omega\in \Sym^{1,1}(\mf g)$, let $\{v_i\}$ be an orthonormal basis of $\mf g$, diagonalizing $\Omega$:
\[
\Omega=\sum_i \lambda_i v_i\otimes\bar v_i.
\]
We define two important quadratic operations on the space $\Sym^{1,1}(\mf g)$. 
\begin{itemize}
	\item [$\mathbf{\Omega^2}$:]
	Metric $g$ induces the isomorphism $\iota_g\colon \Omega\mapsto R^\Omega$, mapping $\Omega$ to the corresponding self-adjoint operator $R^\Omega\colon\mf g\to\mf g$. Define $\Omega^2:= \iota_g^{-1}((R^\Omega)^2)$. In the basis $\{v_i\}$ the square of $\Omega$ is given by
	\[
	\Omega^2:=\sum_i \lambda_i^2 v_i\otimes\bar v_i.
	\]	
	Note that $\Omega^2$ is positive semidefinite, i.e., 
	$\la \Omega^2,u\otimes\bar u\ra_\tr\ge 0$ for any $u\in\mf g$. Moreover, $\la \Omega^2,u\otimes\bar u\ra_\tr=0$ if and only if $u\in\mathrm{Ker}\,\Omega$.
	\item [$\mathbf{\Omega^\#}$:] For $v_1\otimes\bar{ w_1},v_2\otimes \bar{w_2}\in \mf g\otimes\bar{\mf g}$ define
	\begin{equation}\label{eq:sharp}
	(v_1\otimes\bar{ w_1})\#(v_2\otimes \bar{w_2})=[v_1, v_2]\otimes\bar{[w_1,w_2]}.
	\end{equation}
	This map gives rise to a bilinear operation $\#\colon  \Sym^{1,1}(\mf g)\otimes \Sym^{1,1}(\mf g)\to \Sym^{1,1}(\mf g)$. Let $\Omega^\#:=1/2(\Omega\#\Omega)$ be the $\#$-square of $\Omega$. In the basis $\{v_i\}$ the $\#$-square of $\Omega$ is given by
	\[
	\Omega^\#=\sum_{i<j}\lambda_i\lambda_j[v_i, v_j]\otimes\bar{[v_i,v_j]}.
	\]
\end{itemize}

Operation $\Omega^\#$ was introduced by Hamilton in~\cite{ha-86}, while studying the evolution equation for the Riemannian curvature tensor under the Ricci flow. In~\cite{us-17}, we used this operation for an arbitrary Lie algebra $\mf g$ to study the HCF on complex homogeneous manifolds $G/H$. Note that $\Omega^\#$ does not depend on the choice of metric on $\mf g$. The following proposition provides coordinate expressions for $\Omega^2$ and $\Omega^\#$.

\begin{proposition}\label{prop:squares_coordinates}
	Let $\{e_m\}$ be a basis of $V$ and $\{\epsilon^m\}$ be the dual basis. For an element $\Omega\in\Sym^{1,1}(\mf g)$
	\[
	\Omega=\Omega_{i\bar j}^{\ \ \bar l k}(e_k\otimes\epsilon^i)\otimes\bar{(e_l\otimes\epsilon^j)}
	\]
	we have
	\begin{equation}
	\begin{split}
		(\Omega^\#)_{i\bar j}^{\ \ \bar l k}&=
		\Omega_{p\bar n}^{\ \ \bar l k}\Omega_{i\bar j}^{\ \ \bar n p}-
		\Omega_{p\bar j}^{\ \ \bar n k}\Omega_{i\bar n}^{\ \ \bar l p},\\
		(\Omega^2)_{i\bar j}^{\ \ \bar l k}&=g^{m\bar n}g_{p\bar s}\Omega_{i\bar n}^{\ \ \bar s k}\Omega_{m\bar j}^{\ \ \bar l p}.
	\end{split}
	\end{equation}
\end{proposition}
\begin{proof}
	For $e_m\otimes\epsilon^n$, $e_p\otimes\epsilon^s\in\mf g$ we have
	\[
		[e_m\otimes\epsilon^n, e_p\otimes\epsilon^s]=\delta_p^ne_m\otimes\epsilon^s-\delta_m^se_p\otimes\epsilon^n.
	\]
	Therefore 
	\begin{equation}
	\begin{split}
		\Omega\#\Omega=&\Omega_{i\bar j}^{\ \ \bar l k}\Omega_{a\bar b}^{\ \ \bar d c}[e_k\otimes\epsilon^i,e_c\otimes\epsilon^a]\otimes\bar{[e_l\otimes\epsilon^j,e_d\otimes\epsilon^b]}\\=&
		\Omega_{i\bar j}^{\ \ \bar l k}\Omega_{a\bar b}^{\ \ \bar d c}
		(\delta_c^ie_k\otimes\epsilon^a-\delta^a_ke_c\otimes\epsilon^i)\otimes
		\bar{(\delta_d^j e_l\otimes\epsilon^b-\delta_l^be_d\otimes\epsilon^j)}.
	\end{split}
	\end{equation}
	After expanding the Kronecker $\delta$'s we get the expression for $\Omega^\#$.
	
	In coordinates, $R^\Omega$ is given by
	\[
	R^\Omega(e_m\otimes\epsilon^p)=\Omega_{i\bar j}^{\ \ \bar l k}g^{p\bar j}g_{m\bar l}(e_k\otimes\epsilon^i).
	\]
	Therefore
	\[
	(R^\Omega)^2(e_m\otimes\epsilon^p)=\Omega_{n\bar j}^{\ \ \bar l s}\Omega_{i\bar r}^{\ \ \bar q k}g^{p\bar j}g_{m\bar l}g_{s\bar q}g^{n\bar r}(e_k\otimes\epsilon^i),
	\]
	which implies the stated formula for $\Omega^2$.
\end{proof}

\subsection{Hermitian curvature flow}
Let $(M,g,J)$ be an Hermitian manifold. Consider an evolution equation for the Hermitian metric $g=g(t)$.
\begin{equation}\label{eq:HCF}
\pd_t g_{i\bar j}=-S^{(2)}_{i\bar j}-Q_{i\bar j},
\end{equation}
where $S^{(2)}_{i\bar j}=g^{m\bar n}\Omega_{m\bar ni\bar j}$ is the \emph{second Chern-Ricci curvature} and $Q_{i\bar j}=\frac{1}{2}g^{m\bar n}g^{p\bar s}T_{mp\bar j}T_{\bar n\bar s i}$ is a quadratic torsion term. Flow~\eqref{eq:HCF} is a member of the family of Hermitian Curvature Flows, introduced by Streets and Tian in~\cite{st-ti-11}. It is proved in~\cite{st-ti-11} that all these flows are defined by strictly parabolic equations for $g$, and, hence, admit short-time solutions. The particular flow~\eqref{eq:HCF} was first considered by the author in~\cite{us-16} (see also~\cite{us-17}). Further we refer to the flow~\eqref{eq:HCF} as the HCF.

With the HCF, the curvature tensor $\Omega$ also evolve along a nonlinear heat-type equation. Precise form of this equation was computed in~\cite{us-16}. Before stating it, let us recall the notion of~\emph{torsion-twisted} connections.

\begin{definition}[Torsion-twisted connections]\label{def:torsion_twisted_connection}
	We define $\n^T, \n^{T^\sharp}$ to be two \emph{torsion-twisted} connections on $TM$ given by the identities
	\begin{equation}\label{eq:torsion_twisted_connection}
	\begin{split}
	&{\n^T}_X Y=\n_X Y - T(X,Y),\\
	&\n^{T^\sharp}_X Y=\n_X Y+g(Y,T(X,\cdot))^\sharp,
	\end{split}
	\end{equation}
	where $X, Y$ are sections of $T M$, $\n$ is the Chern connection and $\sharp\colon T^*M\to TM$ is the musical isomorphism induced by $g$.
	Equivalently, in the coordinates, for a vector field $\xi=\xi^p \frac{\pd}{\pd z^p}$ one has
	\begin{equation}
	\begin{split}
	&\n^T_i \xi^p=\n_i\xi^p - T^p_{ij}\xi^j, \quad
	\n^T_{\bar i}\xi^p=\n_{\bar i}\xi^p,\\
	&\n^{T^\sharp}_i \xi^p=\n_i\xi^p, \quad
	\n^{T^\sharp}_{\bar i}\xi^p=\n_{\bar i}\xi^p + g^{p\bar s}T_{\bar i\bar s j}\xi^j.
	\end{split}
	\end{equation}
	Both connections preserve the operator of almost complex structure; $\n^T$ is compatible with the holomorphic structure, i.e., $(\n^T)^{0,1}=\bar\pd$.
\end{definition}

\begin{remark}\label{rk:torsion_connection_duality}
	In general, metric $g$ is not preserved by either of the connections $\n^T$, $\n^{T^\sharp}$. However, $g$ is parallel with respect to the connection $\n^T\otimes\n^{T^\sharp}$, i.e.,
	for any vector fields $X,Y,Z\in \Gamma(T M)$ we have
	\[
	X\!\cdot\!g(Y,Z)=g(\n^{T}_X Y, Z) + g(Y,\n^{T^\sharp}_{X} Z).
	\]
	In other words, $\n^{T^\sharp}$ is \emph{dual conjugate} to ${\n^T}$ via $g$.
\end{remark}

\begin{definition}[Torsion-twisted connection on the space of curvature tensors]\label{def:torsion_twisted_laplace}
	\emph{Torsion-twisted connection} $\widetilde\n$ on the space $\Lambda^{1,0}M\otimes\Lambda^{0,1}M\otimes\Lambda^{1,0}M\otimes\Lambda^{0,1}M$ acts as $\n^T$ on the first two factors and as $\n^{T^\sharp}$ on the last two factors.
	
	The Laplacian of this connection is defined as
	\[
	\widetilde\Delta:=\frac{1}{2}\sum_{i}(\widetilde\n_{e_i}\widetilde\n_{\bar{e_i}}+\widetilde\n_{\bar{e_i}}\widetilde\n_{e_i}),
	\]
	where local holomorphic fields $\{e_i\}$ form a unitary frame of $T_m^{1,0}M$, $m\in M$. Similarly we define Laplacian $\Delta^T$ for connection $\n^T$.
\end{definition}

Definition~\ref{def:torsion_twisted_laplace} helps to write down the evolution equation for $\Omega$ under the HCF. The next proposition is a compilation of Propositions 3.1, 3.4, and 3.9 of~\cite{us-16}.
\begin{proposition}\label{prop:omega_evolution}
	Let $g=g(t)$ be the solution to the HCF~\eqref{eq:HCF}. Then $\Omega=\Omega_{i\bar jk\bar l}(t)$ satisfies equation
	\begin{equation}\label{eq:omega_evolution}
	\begin{split}
	\pd_t\Omega_{i\bar j k\bar l}=&
	\widetilde\Delta\Omega_{i\bar j k\bar l}+
	\frac{1}{2}g^{m\bar n}g^{p\bar s}\n_i T_{mp\bar l}\n_{\bar j}T_{\bar n\bar sk}+\\
	&+g^{m\bar n}g^{p\bar s}(
		\Omega_{i\bar j m\bar s}\Omega_{p\bar n k\bar l}-
		\Omega_{i\bar n k\bar s}\Omega_{p\bar j m\bar l}+
		\Omega_{i\bar np\bar l}\Omega_{m\bar j k\bar s}
		)-\\
		&-\frac{1}{2}g^{p\bar s}(
			S^{(2)}_{i\bar s}\Omega_{p\bar j k\bar l}+
			S^{(2)}_{p\bar j}\Omega_{i\bar s k\bar l}+
			S^{(2)}_{k\bar s}\Omega_{i\bar j p\bar l}+
			S^{(2)}_{p\bar l}\Omega_{i\bar jk\bar s}
		)-
		g^{p\bar s}(
			Q_{k\bar s}\Omega_{i\bar j p\bar l}+
			Q_{p\bar l}\Omega_{i\bar j k\bar s}
		)+\\
		&+\frac{1}{2}g^{m\bar n}(
			-\n_{\bar n}T_{mi}^p\Omega_{p\bar jk\bar l}
			-\n_m T_{\bar n\bar j}^{\bar s}\Omega_{i\bar sk\bar l}
			+g^{p\bar s}\n_{\bar n}T_{mp\bar l}\Omega_{i\bar j k\bar s}
			+g^{p\bar s}\n_mT_{\bar n\bar s k}\Omega_{i\bar j p\bar l}
		).
	\end{split}
	\end{equation}
\end{proposition}

In~\cite{us-16} one was interested in the equation for $\pd_t\Omega$ only up to a \emph{first order variation} of $\Omega$ in its arguments. Equation~\eqref{eq:omega_evolution} keeps track of the precise expression for this first order variation. In the form~\eqref{eq:omega_evolution} the evolution equation for $\Omega$ looks unstructured and messy. 
Using Proposition~\ref{prop:omega_evolution} and the the algebraic operations on $\Sym^{1,1}(\mf g)$, we can considerably simplify it.

\begin{proposition}\label{prop:omega_evolution_clear}
	Let $g=g(t)$ be the solution to the HCF~\eqref{eq:HCF}. Then $\Omega\in\Sym^{1,1}(\End(T^{1,0}M))$, $\Omega=\Omega_{i\bar j}^{\ \ \bar l k}(t)$ satisfies equation
	\begin{equation}\label{eq:R_evolution}
		\pd_t \Omega=
		\Delta^T \Omega+
		\Omega^2+
		\Omega^\#+
		D(\n T)+
		\ad_v \Omega,
	\end{equation}
	where
	\begin{enumerate}
		\item[(a)] $D(\n T)_{i\bar j}^{\ \ \bar l k}=1/2
		g^{m\bar n}g^{p\bar s}\n_i T_{mp}^k\n_{\bar j}T_{\bar n\bar s}^{\bar l}$
		\item[(b)] $v_a^b=-\cfrac{1}{2}S^{(4)}_{a\bar s} g^{b\bar s}.$
	\end{enumerate}
\end{proposition}
\begin{proof}
	The proof of this proposition is a matter of straightforward computations. We will not provide these computations in details, but explain the origin of all the summands in~\eqref{eq:R_evolution}.
	
	Recall that $\Omega_{i\bar j}^{\ \ \bar l k}$ is obtained from $\Omega_{i\bar jk\bar l}$ by rising the last two indices. Since the metric $g$ is not preserved by $\n^T$ and $\n^{T^\sharp}$, the derivatives
	\[
	\widetilde{\n} \Omega_{i\bar j}^{\ \ \bar l k}\mbox{ and } 	(\widetilde{\n}\Omega_{i\bar j m\bar n})g^{m\bar l}g^{k\bar n},
	\]
	in general, do not coincide. However, since $\n^T$ and $\n^{T^\sharp}$ are $g$-dual to each other, we have
	\[
	\n^T \Omega_{i\bar j}^{\ \ \bar l k}=	(\widetilde{\n}\Omega_{i\bar j m\bar n})g^{m\bar l}g^{k\bar n}.
	\]
	This explains the presence of the Laplacian $\Delta^T$ in~\eqref{eq:R_evolution}.
	
	Terms quadratic in $\Omega$ in equation~\eqref{eq:omega_evolution} give rise to summands $\Omega^2$ and $\Omega^\#$. It can be easily deduced from Proposition~\ref{prop:squares_coordinates}.
	
	Term $D(\n T)$ trivially comes from the corresponding $\n T$-quadratic term in~\eqref{eq:omega_evolution}.
	
	Finally, the summand $\ad_v\Omega$ comes from the summands of~\eqref{eq:omega_evolution} linear in $\Omega$ plus the terms involving metric derivative: $\Omega_{i\bar j m\bar n}(\pd_t g)^{m\bar l}g^{k\bar n}$ and $\Omega_{i\bar j m\bar n}g^{m\bar l}(\pd_t g)^{k\bar n}$.
\end{proof}

\begin{remark}
	Modulo terms $D(\n T)$ and $\ad_v\Omega$ equation \eqref{eq:R_evolution} coincides with the evolution equation for the Riemannian curvature  under the Ricci flow~\cite{ha-86}.
	This is the only member of a general Streets-Tian's family of flows for which we were able to obtain such a nice evolution equation. It would be interesting to find similar expressions for other versions of the HCF.
\end{remark}

For a subset of the space of algebraic curvature tensors, $X\subset \Sym^{1,1}(\End(T^{1,0}M))$, we write $\Omega\in X$ ($\Omega$ \emph{belongs to $X$}), if $\Omega_m\in X_m $ at any point $m\in M$. We call such an $X$ a~\emph{curvature condition} and say that $\Omega$ \emph{satisfies} $X$.

Let $g=g(t)$ be a solution to the HCF on $M$. Tensor $\Omega^{g(t)}$ satisfies a nonlinear heat-type equation~\eqref{eq:R_evolution} and our aim is to find \emph{invariant curvature conditions} for $\Omega$ under this equation. That are subsets $X\subset \Sym^{1,1}(\End(T^{1,0}M))$ s.t, $\Omega^{g(t)}$ satisfies $X$ for $t>0$, provided $\Omega^{g(0)}$ satisfies $X$. A general approach to this kind of problems was developed by Hamilton in his seminal paper~\cite{ha-86}. This approach is based on \emph{maximum principle for tensors}. In the next section we explain how to modify this principle, so that it will become applicable to equation~\eqref{eq:R_evolution}.

\section{Hamilton's maximum principle}\label{sec:max_principle}

In this section we prove a modification of Hamilton's maximum principle for tensors. Let us start with recalling this principle in its original form. Let $M$ be a closed smooth manifold with a Riemannian metric $g$, and let $E\to M$ be a vector bundle equipped with a metric $h$ and a metric connection $\n^E$. With the use of the Levi-Civita connection we extend $\n^E$ to a connection on $\Lambda^1 M\otimes E$. The Laplacian $\Delta^E\colon \Gamma(M,E)\to\Gamma(M,E)$ is defined as
\[
\Delta^E s:=\tr_g(\n^E\circ\n^E(s)).
\]
Let $\phi(f)$ be a smooth vertical vector field on the total space of $E$. We are interested in the short-time behavior of the solutions to a nonlinear parabolic equation for $f\in C^\infty(M\times[0,\epsilon),\R)$
\begin{equation}\label{eq:PDE_orig}
\frac{d f}{dt}=\Delta^E f+\phi(f),
\end{equation}
where the background data $(h,g,\n^E,\phi)$ is allowed to depend smoothly on $t$.

Recall that a \emph{support} functional for a closed convex set $Y\in\R^N$ at a boundary point $y\in\pd Y$ is a linear function $\alpha\colon \R^N\to \R$, such that $\la\alpha,y\ra\ge\la\alpha,y'\ra$ for any $y'\in Y$. The set of support functionals at $y\in\pd Y$ forms a nonempty closed convex cone in $(\R^N)^*$. The set of support functionals of the unit length (with respect to an underlying metric) will be denoted $\mc S_y$.

Let $X\subset E$ be a subset of the total space of $E$ satisfying the following properties
\begin{enumerate}
	\item[(P1)] $X$ is closed;
	\item[(P2)] the fiber $X_m=X\cap E_m$ over any $m\in M$ is convex;
	\item[(P3)] $X$ is invariant under the parallel transport induced by $\n^E$;
	\item[(P4)] For any boundary point $f\in\pd X_m$ and any support functional $\alpha\in \mc S_f\subset  E_m^*$ at $x$ we have $\la\alpha,\phi(x)\ra\le 0$.
\end{enumerate}
Assume that the initial data $f_0$ lies in $X$, i.e, $f_0(m)\in X_m$ for any $m\in M$. Hamilton's maximum principle states that the set $X$ is invariant under the PDE~\eqref{eq:PDE_orig}, i.e., $f(m,t)$ remains in $X_m$ for $t>0$, as long, as the equation is solvable. Specifically, we have the following results.
\begin{theorem}\label{thm:max_principle_pde_orig}
	If for every fiber $E_m$, $m\in M$, the solutions of the ODE
	\[
		\frac{df}{dt}=\phi(f)
	\]
	remain in $X_m\subset E_m$, then the solutions of the PDE~\eqref{eq:PDE_orig} also remain in~$X$.
\end{theorem}
\begin{lemma}\label{lm:max_principe_ode}
	For a closed convex subset $X_m\subset E_m$ the solution of the ODE
		\[
		\frac{df}{dt}=\phi(f)
		\]
	remains in $X_m$ iff $X_m$ satisfies property (P4).
\end{lemma}

The proof of both results is based on the fact that the invariance of $X$ under PDE/ODE is equivalent to the invariance of all the `half-spaces' $\{x\in X_m\ |\ \la\alpha,x\ra\le\la\alpha,f\ra\}$ for all $f\in\pd X_m$, $\alpha\in\mc S_f$, $m\in M$.
\medskip

In what follows, we will need these results in a slightly more general setup. Namely, we  
\begin{enumerate}
	\item[(a)] do not assume that connection $\n^E$ preserves the bundle metric $h$;
	\item[(b)] in the definition of $\Delta^E$, allow to use any (not necessarily the Levi-Civita) connection $\n^{TM}$ on $TM$ to extend $\n^E$ to a connection on $\Lambda^1M\otimes E$.
\end{enumerate}
Necessity of this generalization comes from the presence of a non-metric Laplacian $\Delta^T$ in equation~\eqref{eq:R_evolution}.
These modifications do not affect neither the setup nor the original proof of Lemma~\ref{lm:max_principe_ode}, since it depends only on the properties of $X$ in each individual fiber $X_m$. Hence, only the proof of Theorem~\ref{thm:max_principle_pde_orig} requires modifications.

\begin{proof}[Proof of Theorem~\ref{thm:max_principle_pde_orig} in a general setup]
	We will go over the Hamilton's proof of the theorem and point the steps requiring the invariance of $h$ under $\n^E$. In each case we provide the necessary modifications to drop this assumption. As in Hamilton's proof, we will use the basic theory of differential inequalities for Lipschitz functions~\cite[\S 3]{ha-86}.
	
	Denote by $|\cdot|$ the length function induced by $h$ on $E$ and $E^*$. Let $f(m,0)=f_0(m)$ be the initial data with $f_0(m)\in X_m$ for any $m\in M$. Let $f(m,t)$ be the solution to the PDE~\eqref{eq:PDE_orig} on $[0,\epsilon]$ and denote by $B_R=\{e\in E\ |\ |e|\le R\}$ the disk bundle of radius $R$ in $E$.
	
	{\bf Step 1.} Without loss of generality we can assume that $X$ is compact and $\phi$ is compactly supported. Indeed, for $R$ large enough $f(m,t)\in B_R$ for any $m\in M$, $t\in[0,\epsilon]$. Consider $\widetilde{X}=X\cap B_{3R}$ and multiply $\phi(f)$ by a cutoff function, which is supported on $B_{3R}$ and equals 1 on $B_{2R}$. Clearly, if the solution of a new equation on $[0,\epsilon]$ stays in $\widetilde{X}$, then the solution of the initial equation stays in $X$. From now $X$ is compact.
	\begin{remark}
		Unlike the situation in the original proof, with the above modification the set $X$ does not remain invariant under $\n^E$, since $h$ and $B_R$ are not preserved by $\n^E$. However, we still have the following local invariance property on $X\cap B_R$
		\begin{itemize}
			\item[(P$3^*$)] There exists $\delta=\delta(g,h,R)>0$ such that for any path $\gamma(\tau)\in M, \tau\in[0,1]$ of length $<\delta$ and any $s\in \pd X_{\gamma(0)}\cap B_R$ the $\n^E$-parallel transport of $s$ along $\gamma$ lies in $\pd X_{\gamma(1)}$. Moreover, this parallel transport carries support functionals to support functionals.
		\end{itemize}
	\end{remark}
	\bigskip
	
	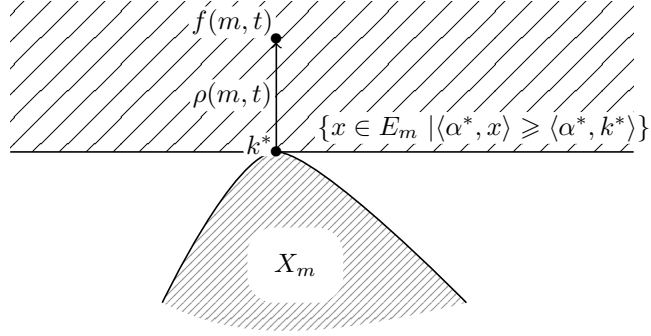
\begin{figure}
		\centering
		\begin{tikzpicture}
		\draw[name path=curve, line width=0.6pt, smooth]plot coordinates {(0,0) (1.5,2) (4,0)};
		\path[name path=line] (0,0) .. controls (1,-0.5) and (2,-0.5) .. (4,0);
		
		\draw[name path=tangent, line width=0.6pt] (-2,2)--(6.2,2);
		
		\fill [pattern=my north east lines] (-2,2) rectangle (6.2,4);
		\node[above right, fill=white,rounded corners=2pt,inner sep=1pt] at (2, 2.1) {$\{x\in E_{m}\ | \la\alpha^*,x\ra\ge\la\alpha^*,k^*\ra\}$};
		\node[left, fill=white,rounded corners=2pt,inner sep=1pt] at (1.5, 2.1) {$k^*$};
		\node at (1.5,2) {\textbullet};
		\node[above left, fill=white,rounded corners=2pt,inner sep=1pt] at (1.5, 3.5) {$f(m,t)$};
		\node at (1.5,3.5) {\textbullet};
		\draw [->,line width=0.6pt] (1.5,2.05) -- (1.5,3.45);
		\node[left, fill=white,rounded corners=2pt,inner sep=1pt] at (1.49, 2.75) {$\rho(m,t)$};
		\begin{pgfonlayer}{bg}
		\fill [pattern=north east lines,pattern color=black!40,
		intersection segments={
			of=curve and line,
			sequence={L2--R2}
		}];
		\end{pgfonlayer}
		
		\node [fill=white,rounded corners=10pt,inner sep=10pt] at (1.75,0.5) {$X_{m}$};
		\end{tikzpicture}
	\caption{Definition of $\rho(m,t)$.}\label{fig:rho}
	\end{figure}
	
	{\bf Step 2.} For a fixed $m\in M$, and $t\in [0,\epsilon]$ define
	\[
	\rho(m,t)=\sup\{\la\alpha,f(m,t)-k\ra\},
	\]
	where the supremum is taken over $k\in\pd X_m$, $\alpha\in \mc S_k$ (i.e., $\alpha$ is a support functional at $k$ and $|\alpha|=1$). Since the domain of this supremum is compact, it is attained at some $\alpha=\alpha^*$, $k=k^*$. By our choice of $R$, point $k^*$ lies in $X_m\cap B_R$.
	If $f(m,t)\not\in X_{m}$, then $\rho(m,t)$ equals the distance from $f(m,t)$ to $\pd X_{m}$ (see Figure~\ref{fig:rho}). Otherwise, if $f(m,t)\in X_m$, then $\rho(m,t)$ equals the negative distance from $f(m,t)$ to $\pd X_{m}$.
	
	Now define
	\[
	\widehat{\rho}(t)=\sup_{m\in M}\rho(m,t).
	\]
	Function $\widehat\rho(t)$ is Lipschitz, and $\widehat\rho(t)\le 0$ (resp.\,$<0$) if and only $f\in X$ (resp.\,$f$ belongs to the interior of $X$). Therefore, to prove Theorem~\ref{thm:max_principle_pde_orig} it is enough to prove that $\widehat\rho(t)\le 0$, provided $\widehat\rho(0)\le 0$. We claim that there exists a constant $C>0$ such that
	\[
	\frac{d\widehat\rho}{dt}\le C|\widehat\rho(t)|.
	\]

	To prove the claim we plug in the definition of $\widehat\rho$ and use~\cite[Lemma 3.5]{ha-86}:
	\[
	\frac{d\widehat\rho}{dt}\le\sup\Bigl\{ \frac{d}{dt}\la\alpha,f(m,t)-k\ra
	\Bigr\},
	\]
	where supremum is taken over $m\in M$, $k\in\pd X_m$, $\alpha\in \mc S_{k}$ such that the maximum of $\la\alpha,f(m,t)-k\ra$ is attained, i.e., $\la\alpha,f(m,t)-k\ra=\widehat\rho(t)$. Together with the equation for $f$ this gives
	\[
	\frac{d\widehat\rho}{dt}\le\sup\{\la\alpha, \Delta^E f+\phi(f)\ra\}=\sup\{\la\alpha, \Delta^E f\ra+\la\alpha,\phi(f)\ra\}.
	\]
	We claim that both summands could be bounded from above by $C|\widehat\rho(t)|$ for some constant $C>0$.

	{\bf Step 2a.} Let $\{e_i\}$ be a $g$-orthonormal frame of $T_mM$. Define $\gamma_i(\tau),i=1,\dots\dim M$, to be the geodesic path of connection $\n^{TM}$ in the direction $e_i$ and denote by $D_i$ the covariant derivative along $\gamma'_i$. Then
	\[
	\Delta^E=\sum_{i=1}^{\dim M} D_i^2.
	\]
	We extend vectors $k\in E_m,\alpha\in E_m^*$ along each of the paths $\gamma_i$ by $\n^E$-parallel transport. By property (P$3^*$), in a small neighborhood of $m\in M$ we still have $k\in\pd X$, and $\alpha$ is a support functional at $k$. Note, however, that in order to get an element in $\mc S_k$ over a point $m_0\neq m$, we need to normalize $\alpha$, since $\n^E$ does not preserve metric $h$; so	$\alpha/|\alpha|\in\mc S_k$. Since point $m\in M$, and the corresponding vectors $k\in\pd X_m$, $\alpha\in \mc S_k$ are chosen in such a way that $\la\alpha,f(m,t)-k\ra$ attains its maximum~--- $\widehat\rho(t)$, the function $\Phi_i(\tau):=\la\alpha/|\alpha|,f(\gamma_i(\tau),t)-k\ra$ is maximal at $\tau=0$. Therefore
	\[
	0=\Phi'_i(0)=(D_i |\alpha|^{-1})\widehat\rho(t)+D_i\la\alpha,f-k\ra;
	\]
	\[
	0\ge \Phi_i''(0)=(D^2_i|\alpha|^{-1})\widehat\rho(t)+2(D_i|\alpha|^{-1})D_i\la\alpha,f-k\ra+\la\alpha,D_i^2f\ra.
	\]
	With the use of the first equation we can rewrite the inequality as
	\[
	\la\alpha,D_i^2f\ra\le -(D^2_i|\alpha|^{-1})\widehat\rho(t)+2(D_i|\alpha|^{-1})^2\widehat\rho(t).
	\]
	Let $C'$ be an upper bound for $\bigl|2(D_i|\alpha|^{-1})^2-D_i^2|\alpha|^{-1}\bigr|$ over $m\in M$, $\alpha\in \{\alpha\in E_m^*\ |\ |\alpha|=1\}$, $e_i\in \{v\in T_mM\ |\ |v|=1\}$. Summing the inequality above over $i=1,\dots,\dim M$, we deduce the inequality
	\[
	\la\alpha,\Delta^Ef\ra\le C|\widehat\rho(t)|,
	\]
	for $C=C'\dim M$, as required.
	\begin{remark}
		In the original proof, the bundle connection $\n^E$ preserves $h$, hence for the $\n^E$-parallel extension of $\alpha$ we have $|\alpha|\equiv 1$, so $\alpha\in\mc S_k$. In particular, we could take $C=0$.
	\end{remark}
	
	{\bf Step 2b.} Recall that $\alpha\in\mc S_k$, therefore by property (P4) $\la\alpha,\phi(k)\ra\le 0$. Hence we have
	\[
	\la\alpha,\phi(f)\ra=\la\alpha,\phi(k)\ra+\la\alpha,\phi(f)-\phi(k)\ra\le\la\alpha,\phi(f)-\phi(k)\ra\le C|f-k|=C|\widehat\rho(t)|,
	\]
	where $C$ is a generic constant bounding the norm of the derivative of $\phi\colon E\to E$.
	
	{\bf Step 3.} Lipschitz function $\widehat\rho(t)$ satisfies initial condition $\widehat\rho(0)\le 0$ and differential inequality $d\widehat\rho/dt\le C|\widehat\rho|$. By a general result~\cite{ha-86}, it implies $\widehat\rho(t)\le 0$ for $t\ge 0$. This is equivalent to the required invariance: $f(m,t)\in X_m$ for any $m\in M$, $t\ge 0$. This proves Theorem~\ref{thm:max_principle_pde_orig}.
	
	With the same reasoning, we can prove a bit more. If $\widehat\rho(0)<0$, then $\widehat\rho(t)\le\widehat\rho(0)e^{-Ct}<0$ for $t\ge0$. Therefore, if $f(m,0)$ lies in the interior of $X$ for all $m\in M$, then the same is true for $f(m,t)$, $t>0$. So, the interior of $X$ is also preserved by the PDE~\eqref{eq:PDE_orig}.
\end{proof}

Theorem~\ref{thm:max_principle_pde_orig} allows to construct many invariant sets $X$ for certain PDEs of the form~\eqref{eq:PDE_orig}. In practice, conditions (P1), (P2), (P3) are satisfied automatically for a wide range of subsets $X\subset E$, while (P4) is the most essential and difficult to verify. In the next section we apply these results to the evolution equation of the Chern curvature under the HCF.

\section{Invariant sets of curvature operators}\label{sec:main}
By the philosophy of Hamilton's maximum principle in order to find invariant sets for a heat-type equation $\pd_t f=\Delta f+\phi(f)$ one has to study an ODE $\pd_t f=\phi(f)$. Following this idea, in this section we study an ODE on the space of algebraic curvature tensors $\Sym^{1,1}(\mf g)$ given by the zero-order part of the equation~\eqref{eq:R_evolution}, where $\mf g=\End(V)$ as in Section~\ref{subsec:algebra}. We construct a family of invariant sets for this ODE by verifying property (P4). 

\subsection{ODE-invariant sets}
Let $\Omega\in\Sym^{1,1}(\mf g)$. Consider an ODE for $\Omega=\Omega(t)$

\begin{equation}\label{eq:ODE_R}
\pd_t \Omega=\Omega^2+\Omega^\#+\ad_v \Omega +AA^*,
\end{equation}
where
\begin{enumerate}
	\item $v$ is an element of $\mf g$;
	\item $AA^*=\sum_i a_i\otimes \bar{a_i}$, for some collection of vectors $\{a_i\in\mf g\}$.
\end{enumerate}
Both $v$ and $A$ are allowed to depend on time. Following the notations of Section~\ref{sec:max_principle}, denote $\phi(\Omega):=\Omega^2+\Omega^\#+\ad_v \Omega +AA^*$.  We describe a family of convex subsets of $\Sym^{1,1}(\mf g)$, for which we are aiming to prove invariance under~\eqref{eq:ODE_R}, and, eventually, the invariance under the HCF.
\medskip

Let $S\subset \mf g$ be a subset invariant under the adjoint action of $G=GL(V)$ and let $F\colon \mf g\to\R$ be a continuous function satisfying the following properties.

\begin{itemize}
	\item[$(\star_1)$] $F$ is $\Ad\,G$-invariant. Since diagonalizable matrices are dense in $\mf g$, $F$ can be though of as a symmetric function in the eigenvalues $\{\mu_i\}$ of $s\in \mf g$;
	\item[($\star_2$)] For any sequence $s_i\in \mf g$ and any $\lambda_i\searrow 0$, such that $\lambda_is_i\to s$, there exists a finite limit
	\[
	\lim_{i\to\infty } F(s_i)\lambda_i^2,
	\]
	and its value depends only on $s$. We denote this limit by $F_\infty(s)$.
\end{itemize}
\begin{definition}\label{def:nice}
	Continuous function $F\colon\mf g\to \R$ satisfying properties $(\star_1)$ and $(\star_2)$ is called~\emph{nice}. There is a function $F_\infty\colon\mf g\to \R$ attached to any nice function.
\end{definition}
Examples of nice $F$ are $F(s)=a|\tr s|^2+b$ with the corresponding limit $F_\infty(s)=a|\tr s|^2$ and $F(s)=\sum_{\mu_i\in\mathrm{spec}(s)} |\mu_i|$ with $F_\infty(s)\equiv 0$. In most of the examples below the only reasonable choice is $F\equiv 0$.

Given a tuple $(S,F)$ we define a subset of $\Sym^{1,1}(\mf g)$:
\[
C(S,F):=\{\Omega\in\Sym^{1,1}(\mf g)\ |\ \la\Omega,s\otimes\bar s\ra_\tr\ge F(s) \mbox{ for all }s\in S\}.
\]
As the intersection of closed halfspaces, the set $C(S,F)$ is closed and convex. Since $S$ and $F$ are $\Ad\,G$-invariant, set $C(S,F)$ is also invariant under the induced action of $G$. We claim that $C(S,F)$ is preserved by ODE~\eqref{eq:ODE_R}. 
\begin{theorem}\label{thm:ode_invariance}
	The set $C(S,F)\subset \Sym^{1,1}(\mf g)$ is closed, convex and satisfies property (P4) for ODE~\eqref{eq:ODE_R}. In particular, by Lemma~\ref{lm:max_principe_ode} set $C(S,F)$ is invariant under this ODE.
\end{theorem}

Let us first prove a lemma. In this lemma we do not assume that $F$ is continuous. Its proof follows the lines of~\cite{wi-13}.

\begin{lemma}\label{lm:property_P4}
	If $\Omega\in C(S,F)$ and $u\in S$ is such that $\la\Omega,u\otimes\bar u\ra_\tr=F(u)$, then $\la\Omega^2+\Omega^\#+\ad_v \Omega+AA^*,u\otimes\bar u\ra_\tr\ge 0$.
\end{lemma}
\begin{proof}
	Recall that by definition of $\Omega^2$ we have
	$\la\Omega^2,u\otimes\bar u\ra_\tr\ge 0$. Similarly $\la AA^*,u\otimes \bar u\ra_\tr=\sum_{i} |\tr(a_iu)|^2\ge 0$. Hence, it remains to prove that $\la\Omega^\#+\ad_v \Omega,u\otimes\bar u\ra_\tr\ge 0$. We claim, that
	\begin{enumerate}
		\item[(C1)] $\la\ad_v\Omega,u\otimes\bar u\ra_\tr=0$,
		\item[(C2)] $\la\Omega^\#,u\otimes\bar u\ra_\tr\ge 0$.
	\end{enumerate}
	Proof of both claims is based on the variation of $\la\Omega,u\otimes\bar u\ra_\tr$ in $u$.
	
	Fix an element $x\in\mf g$ and let $u(\tau)=\exp(\tau\,\ad_x)u$ be an orbit of $u$ induced by a 1-parameter subgroup of $\Ad\,G$. By the invariance of $S$ under the adjoint action, we have $u(\tau)\in S$. Therefore by the definition of $C(S,F)$, the function
	\[
	\Psi(\tau):=\la\Omega,u(\tau)\otimes \bar{u(\tau)}\ra_\tr
	\]
	is bounded below by $F(u)$ and, by our choice of $u$,
	attains this minimum at $\tau=0$. It implies that $\Psi'(0)=0$ and $\Psi''(0)\ge 0$. Specifically,
	\begin{equation}\label{eq:proof_ODE_ineq}
	\begin{split}
		&\la\Omega,\ad_x u\otimes\bar u+u\otimes\bar{\ad_x u}\ra_\tr=0,\\
		&\la\Omega,\ad_x u\otimes \bar{\ad_x u}\ra_\tr+
		\la\Omega,\ad_x\ad_x u\otimes \bar u\ra_\tr+
		\la\Omega,u\otimes \bar{\ad_x \ad_x u}\ra_\tr\ge 0.
	\end{split}
	\end{equation}
	The first identity is equivalent to $\la\ad_x\Omega,u\otimes\bar u\ra_\tr=0$ for all $x\in\mf g$, implying the vanishing of (C1). Now let us prove (C2). After summing up the second line of~\eqref{eq:proof_ODE_ineq} for $x$ and $\sqrt{-1}x$ we arrive at
	\begin{equation}\label{eq:proof_ODE_ineq_basic}
		\la\Omega,\ad_x u\otimes \bar{\ad_x u}\ra_\tr\ge 0.
	\end{equation}
	This inequality holds for any $x\in\mf g$, thus Hermitian form $Q^\Omega(x,\bar x):=\la\Omega,\ad_x u\otimes \bar{\ad_x u}\ra_\tr=\la\Omega,\ad_u x\otimes \bar{\ad_u x}\ra_\tr$ is positive semidefinite.
	
	Let us choose a basis $\{v_i\}_{i=1}^{N}$ of $\mf g$ such that
	\begin{itemize}
		\item $\{v_i\}_{i=r+1}^N$ is a basis of $\mathrm{Ker}\,\ad_u$, so $Q^\Omega(v_i,\bar{\cdot})=0$ for any $r+1\le i\le N$;
		\item $Q^\Omega(v_i,\bar v_j)=\delta_{ij}\mu_i$, $\mu_i\ge 0$, for $1\le i,j\le r$, or, equivalently, the form $\la\Omega,\cdot\otimes\bar{\cdot}\ra_\tr\Bigl|_{\mathrm{Im}\,\ad u}$ is diagonalized in the basis $\{w_i=\ad_u v_i\}_{i=1}^r$;
		\item $\{w_i=\ad_u v_i\}_{i=1}^r$ is an orthonormal basis of $\mathrm{Im}\,\ad_u$ with respect to the inner Hermitian product $\tr(a b^*)$, $a,b\in\mf g$, where $b^*:={\bar b}^\tr$ is the transposed conjugate of $b$ in some fixed basis of $V$. In other words, $\tr(w_iw_j^*)=\delta_{ij}$.
	\end{itemize}
	
	Let $\Omega=\sum_{i,j=1}^N a_{i\bar j}v_i\otimes \bar{v_j}$ be the expression for $\Omega$ in this basis. Then 
	\[
	\Omega^\#=\frac{1}{2}\sum_{i,j,k,l=1}^N a_{i\bar j}a_{k\bar l}[v_i,v_k]\otimes\bar{[v_j,v_l]}.
	\]
	Therefore
	\begin{equation}
	\begin{split}
	\la\Omega^\#,u\otimes\bar u\ra_\tr&=
	\frac{1}{2}\sum_{i,j,k,l=1}^N a_{i\bar j}a_{k\bar l}\tr(u[v_i,v_k])\otimes\bar{\tr(u[v_j,v_l])}\\
	&=\frac{1}{2}\sum_{i,j,k,l=1}^r a_{i\bar j}a_{k\bar l}\tr(v_i[u,v_k])\otimes\bar{\tr(v_j[u,v_l])}\\
	&=\frac{1}{2}\sum_{k,l=1}^r a_{k\bar l}Q^\Omega(v_k,\bar{v_l})=\frac{1}{2}\sum_{k=1}^r a_{k\bar k}\mu_k.
	\end{split}
	\end{equation}
	It remains to show that $a_{k\bar k}\ge 0$: 
	\[
	Q^\Omega(w_k^*,\bar{w_k^*})=\sum_{i,j=1}^r
	a_{i\bar j}\tr([v_i, u]w_k^*)\bar{\tr([v_j,u]w_k^*)}=\sum_{i,j=1}^r
	a_{i\bar j}\tr(w_iw_k^*)\bar{\tr(w_jw_k^*)}=
	a_{k\bar k},
	\] hence $a_{k\bar k}\ge 0$, and $\la\Omega^\#,u\otimes\bar u\ra_\tr=\sum_k a_{k\bar k}\mu_k\ge 0$, as required.
\end{proof}
\begin{proof}[Proof of Theorem~\ref{thm:ode_invariance}]
Now we prove property (P4) for ODE~\eqref{eq:ODE_R} and convex set $C(S,F)$. Let $\bar S$ be the closure of $S$. Clearly $C(S,F)=C(\bar S,F)$, so without loss of generality we can assume that $S=\bar S$. Take a point at the boundary of $C(S,F)$:
\[
y\in \pd C(S,F).
\]
We want to describe the set of support functionals for $C(S,F)$ at $y$. Let $\alpha$ be such a functional and take any $w\in \Sym^{1,1}(\mf g)$ such that $\la \alpha,w\ra>0$. Since $\alpha$ is a support functional, for any $\theta>0$ we have $y+\theta w\not\in C(S,F)$, i.e., there exists $s\in S$ (depending on $w$ and $\theta$) such that 
\[
\la y+\theta w,s\otimes \bar s\ra_\tr<F(s).
\]
Let $\theta_i\searrow 0$ be a monotonically decreasing sequence of real numbers. Choose $s_i\in S$ such that the inequality above holds. Since $y\in C(S,F)$, we have
\[
F(s_i)+\la\theta_i w,s_i\otimes \bar s_i\ra_\tr\le \la y+\theta_i w,s_i\otimes \bar s_i\ra_\tr<F(s_i).
\]
There are two options.
\begin{enumerate}
	\item Some subsequence of $|s_i|$ stays bounded. Then after passing to a subsequence we may assume that $s_i\to s\in S$. In this case we have
	\[
	\la w,s\otimes\bar s\ra_\tr\le 0,\quad \la y,s\otimes\bar s\ra_\tr=F(s),\ \mbox{for some } s\in S.
	\]
	\item $|s_i|\to\infty$. Then after passing to a subsequence we may assume that for some $\lambda_i\searrow 0$ the sequence $\lambda_i s_i$ converges to an element $s$ in the set:
	\begin{equation}
	\pd_\infty S:=\{Y\in\mf g\ |\ \mbox{there exists } \lambda_i\searrow 0, s_i\in S\mbox{ with }\lambda_i s_i\to Y\}.
	\end{equation}
	This set is called \emph{the boundary of $S$ at infinity}. From the definition of $F_\infty$ it is clear that $\Omega\in C(\pd_\infty S,F_\infty)$. In this case we have
	\[
	\la w,s\otimes\bar s\ra_\tr\le 0,\quad \la y,s\otimes\bar s\ra_\tr=F_\infty(s),\ \mbox{for some }s\in \pd_\infty S.
	\]
\end{enumerate}
Inequality $\la w,s\otimes\bar s\ra_\tr\le 0$ is valid in both cases for any $w$ such that $\la \alpha,w\ra>0$. Therefore $\alpha$ cannot be separated in $\Sym^{1,1}(\mf g)^*$ by a hyperplane $\la\alpha,w\ra=0$ from the set of functionals
\[
\mc F_y:=\mc F_y^b\cup\mc F_y^\infty,
\]
where
\[
\mc F_y^b=\{-\la\cdot,s\otimes \bar s\ra_\tr\ |\ s\in S \mbox{ s.t.\,} \la y,s\otimes\bar s\ra_\tr=F(s)\},
\]
\[
\mc F_y^\infty=\{-\la\cdot,s\otimes \bar s\ra_\tr\ |\ s\in \pd_\infty S \mbox{ s.t.\,} \la y,s\otimes\bar s\ra_\tr=F_\infty(s)\}.
\]
Hence, $\alpha$ lies in the convex cone spanned by the elements of $\mc F_y$: 
\[
\alpha\in \mathrm{Cone}(\mc F_y).
\]
Thus, in order to verify property (P4), we need to check that $\la\alpha,\phi(\Omega)\ra\le 0$ for all $\alpha$ in $\mc F_y$. For $\alpha\in\mc F_y^b$ this is exactly the statement of Lemma~\ref{lm:property_P4}. For $\alpha\in\mc F_y^\infty$ this is the statement of Lemma~\ref{lm:property_P4} applied to $C(\pd_\infty S,F_\infty)$ (at this point continuity of $F_\infty$ is not required). This proves property (P4), and, by Lemma~\ref{lm:max_principe_ode}, the invariance of $C(S,F)$ under ODE~\eqref{eq:ODE_R}.
\end{proof}

\subsection{PDE-invariant sets}
With the results of previous subsection we can turn back to PDE~\eqref{eq:R_evolution} satisfied by the Chern curvature tensor $\Omega$ under the HCF on $(M,g,J)$.  First we note that the term
\[D(\n T)_{i\bar j}^{\ \ \bar l k}=1/2
g^{m\bar n}g^{p\bar s}\n_i T_{mp}^k\n_{\bar j}T_{\bar n\bar s}^{\bar l}
\]
is of the form $AA^*$ (see~\eqref{eq:ODE_R}) for
\begin{equation}\label{eq:A_T}
\{a_{(mp)}\}=\{\n_i T_{mp}^ke_k\otimes\epsilon^i\ |\ m<p\}
\end{equation}
in some orthonormal basis $\{e_i\}$.
Hence the zero order part of the PDE for $\Omega$ is a specialization of the right hand side of ODE~\eqref{eq:ODE_R}.
\begin{equation}\label{eq:R_evolution_2}
\pd_t \Omega=
\Delta^T \Omega+
\Omega^2+
\Omega^\#+
D(\n T)+
\ad_v \Omega
\end{equation}
Let $V=\C^{\dim M}$. For every $m\in M$ choose a linear isomorphism $V\simeq T_m^{1,0}M$, with the corresponding isomorphism $\mf g\simeq\End(T_m^{1,0}M)$. Then any $G=GL(V)$-invariant $S\subset \mf g$ canonically corresponds to $S_m\subset \mf \End(T_m^{1,0})$. The set $C(S,F):=\cup_{m\in M} C(S_m,F)\subset \Sym^{1,1}(\End(T^{1,0}M))$ is closed, convex and satisfies (P4). It also satisfies property (P3), since $C(S,F)\subset \End(T^{1,0}M)$ is invariant under the adjoint action of $GL(T^{1,0}M)$, and, therefore is invariant under the action of the holonomy group any connection compatible with $J$. Hence, we can apply Theorem~\ref{thm:max_principle_pde_orig} and conclude that $C(S,F)$ is invariant under~\eqref{eq:R_evolution_2}, proving our main result.
\begin{theorem}\label{thm:main}
	For any $Ad\,G$-invariant subset $S\subset \mf g$ and any nice $F\colon \mf g\to\R$ the curvature condition $C(S,F)$ is preserved by the HCF~\eqref{eq:HCF}.
\end{theorem}

\section{Examples}\label{sec:examples}
Let us provide some specific examples of curvature conditions preserved by the HCF. In most of the examples $F\equiv 0$.

\begin{example}[Dual-Nakano non-negativity]
Recall that the Chern curvature $\Omega$ 
is dual-Nakano non-negative, if it represents a non-negative element in $\Sym^{1,1}(\End(T^{1,0}M))$, see Remark~\ref{rm:nakano}. Choose $S=\End(T^{1,0}M)$. Then the cone $C(S,0)$ is the set of dual-Nakano non-negative curvature tensors. By Theorem~\ref{thm:main} this set is preserved by the HCF.
\end{example}

Manifolds admitting Dual-Nakano non-negative hermitian metrics are rather scarce. However, as demonstrates the following example, such metrics exist on all complex homogeneous manifolds.

\begin{example}[Dual-Nakano non-negative metrics on complex homogeneous manifolds]\label{ex:homogeneous}
	Let $M=G/H$ be a complex homogeneous manifold acted on by a complex Lie group $G$. Denote by $\mf g,\mf h$ the corresponding Lie algebras. There is an exact sequence of holomorphic vector bundles
	\[
	0\to \Lambda^{1,0}M\xrightarrow{\iota}\mf g^*\xrightarrow{\pi}\mf h^*\to 0.
	\]
	The \emph{second fundamental form} $\beta\in \Lambda^{1,0}(M,\mathrm{Hom}(\Lambda^{1,0}M,\mf h^*))$ of this exact sequence is given by
	\[
	\beta_\xi(\alpha)=\pi(D_\xi (\iota\alpha))\in\mf h^*,\quad \xi\in T^{1,0}M,\ \alpha\in\Lambda^{1,0}M,
	\]
	where $D$ is the flat connection on the trivialized vector bundle $\mf g^*$. Second fundamental form $\beta$ naturally corresponds to a map $A_\beta\colon\mf h\to\End(T^{1,0}M)$.
	
	Now let $g$ be any Hermitian metric on $\mf g$. It induces an Hermitian metric on $T^{1,0}M$ via projection $\mf g\to T^{1,0}M$. By a standard computation~\cite[\S 14]{de} its curvature $\Omega\in\Sym^{1,1}(\End(T^{1,0}M))$ is given by
	$A_\beta^*A_\beta=\sum_{i}A_\beta(h_i)\otimes\bar{A_\beta(h_i)}$, where $\{h_i\}$ is the orthonormal basis of $\mf h$. This form is clearly non-negative.	
\end{example}
\begin{remark}
	In Example~\ref{ex:homogeneous} for any Hermitian metric $g$ on $\mf g$ we considered the \emph{induced metric} on $M=G/H$. It is proved in~\cite{us-17} that the set of induced metrics is invariant under the HCF.
\end{remark}
\begin{example}[Griffiths non-negativity]
Now we demonstrate preservation of Griffiths non-negativity under the HCF. It was first proved in~\cite{us-16} by adopting the arguments of Mok~\cite{mo-88} and Bando~\cite{ba-84}, who proved the corresponding statement for the K\"ahler-Ricci flow. Below we deduce preservation of Griffiths non-negativity as a particular case of Theorem~\ref{thm:main}. Let us recall the definition.
\begin{definition}
	Chern curvature tensor $\Omega=\Omega_{i\bar j k\bar l}$, considered as a section of $\Lambda^{1,0}M\otimes \Lambda^{0,1}M\otimes \Lambda^{1,0}M\otimes\Lambda^{0,1}M$ is said to be \emph{Griffiths non-negative}, if for any $\xi,\eta\in T^{1,0}M$
	\[
	\Omega(\xi,\bar \xi,\eta,\bar{\eta})\ge 0.
	\]
\end{definition}
In the K\"ahler setting Griffiths positivity is sometimes referred to as positivity of the \emph{holomorphic bisectional curvature}. Some authors use this notion in the Hermitian setting as well, see, e.g.,~\cite{ya-17}. Griffiths positivity implies ampleness of $T^{1,0}M$.

It is easy to see that Chern curvature $\Omega$ considered as a section of $\Sym^{1,1}(\End(T^{1,0}M))$ is Griffiths non-negative if and only if $\Omega\in C(S,0)$, where 
\[
S=\{u\in \End(T^{1,0}M)\ |\ \mathrm{rank}(u)=1\}.
\]
Therefore Griffiths non-negativity is preserved under the HCF.
\end{example}
\begin{example}[Nonnegativity of holomorphic orthogonal bisectional curvature]
\begin{definition}
	Chern curvature tensor $\Omega=\Omega_{i\bar j k\bar l}$, considered as a section of $\Lambda^{1,0}M\otimes \Lambda^{0,1}M\otimes \Lambda^{1,0}M\otimes\Lambda^{0,1}M$ is said to have non-negative \emph{holomorphic orthogonal bisectional curvature}, if for any $\xi,\eta\in T^{1,0}M$ s.t $g(\xi,\bar{\eta})=0$
	\[
	\Omega(\xi,\bar \xi,\eta,\bar{\eta})\ge 0.
	\]
\end{definition}
This curvature condition is a relaxation of Griffiths non-negativity. K\"ahler manifolds admitting non-negative holomorphic orthogonal bisectional curvature were classified by Gu and Zhang~\cite{gu-zha-10}.

Chern curvature $\Omega$ considered as a section of $\Sym^{1,1}(\End(T^{1,0}M))$ has non-negative holomorphic orthogonal bisectional curvature if and only if $\Omega\in C(S,0)$, where
\[
S=\{u\in \End(T^{1,0}M)\ |\ \mathrm{rank}(u)=1,\ \tr(u)=0\}.
\]
Therefore non-negativity is preserved under the HCF.
\end{example}
\begin{example}[Dual $m$-non-negativity] This curvature positivity notion interpolates between Griffiths non-negativity ($m=1$) and dual-Nakano non-negativity $(m=\dim M)$.
\begin{definition}
	Take a number $1\le m\le\dim M$.
	Chern curvature tensor $\Omega\in\Sym^{1,1}(T^{1,0}M)$ is \emph{dual $m$-non-negative} if $\Omega\in C(S_m,0)$, where $S_m=\{u\in\mf g\ |\ \mathrm{rank}(u)=m\}$.
\end{definition}
It is clear from the definition and Theorem~\ref{thm:main} that dual $m$-non-negativity is preserved under the HCF.
\end{example}
\begin{example}[Lower bounds on the second scalar curvature]\label{ex:scalar_curvature}
	It is well-known that under the Ricci flow the lower bound on the scalar curvature is improved, unless the manifold is Ricci-flat. It turns out that the second scalar curvature under the HCF satisfies similar monotonicity.
	Namely, take $S=\{\mathrm{Id}\}\in\End(T^{1,0}M)$. Then for any $q\in\R$ the set $C(S,q)$ is preserved under the HCF. In particular, the infimum of $\la\Omega, \mathrm{Id}\otimes\bar{\mathrm{Id}}\ra_\tr=\Omega_{i\bar j}^{\ \ \bar j i}=\scal$ is nondecreasing.
	
	The same result can be obtained without invoking Hamilton's maximum principle for tensors. Indeed, after contracting equation~\eqref{eq:R_evolution}, we get
	\begin{equation}\label{eq:s_evolution}
	\pd_t \scal=\Delta \scal+|S^{(3)}|^2+\frac{1}{2}|\mathrm{div} T|^2,
	\end{equation}
	where $(\mathrm{div} T)_{jk}=\n_iT_{jk}^i$. The zero-order expression on the right-hand side is non-negative and by the standard maximum principle for parabolic equations the quantity $\inf_M \scal$ is nondecreasing in~$t$.
\end{example}
\section{Applications}\label{sec:applications}
\subsection{Strong maximum principle}\label{sec:strong_maximum_principle}
Hamilton's maximum principle in the form of Theorem~\ref{thm:max_principle_pde_orig} is a variant of a \emph{weak} parabolic maximum principle, i.e., a statement about preservation of a non-strict inequality along a heat-type flow. In many settings a \emph{strong} maximum principle is satisfied. That is a statement characterizing solutions $f(t)$ of \eqref{eq:PDE_orig}, which meet the boundary of a preserved set~$X$ at some $t>0$. We describe a version of the strong maximum principle for Theorem~\ref{thm:main}.

\begin{theorem}\label{thm:br-sc}
	Consider an $\Ad\,G$-invariant $S\subset \mf g$ and a nice function $F\colon \mf g\to \R$. Let $(M,g,J)$ be an Hermitian manifold with metric $g=g(t)$, $t\in[0,\tau)$ evolved along the HCF. Assume that $\Omega^{g(0)}$ satisfies $C(S,F)$. Then for any $t>0$ the set
	\[
	N(t,m):=\{s\in S\ |\ \la\Omega_m^{g(t)},s\otimes\bar s\ra_\tr=F(s)\},\quad m\in M,
	\]
	is invariant under the $\n^T$-parallel transport.	Moreover, if $s_0\in N(t,m)$, then
	\begin{enumerate}
		\item[(a)] $s_0$ belongs to the kernel of $\la\Omega_m^{g(t)},\cdot\otimes \bar{\cdot}\ra_\tr$, in particular, $F(s_0)=0$;
		\item[(b)] $\la \n T(\xi,\eta),s_0\ra_\tr=(s_0)_j^i\n_iT_{kl}^j\xi^k\eta^l=0$ for any $\xi,\eta\in T_m^{1,0}M$.
	\end{enumerate}
\end{theorem}

This theorem is an extension of Brendle and Schoen's~\cite{br-sc-08} strong maximum principle, which was originally proved for the isotropic curvature evolved under the Ricci flow. A general argument in the context of the Ricci flow was given by Wilking~\cite[A.1]{wi-13}. The same proof works for the HCF with minor modifications. In~\cite[Th.\,5.2]{us-16} this argument was used in the case of Griffiths positivity.
\begin{proof}
We may assume that $S$ is an orbit of the $\Ad\,G$-action on $\mf g$, otherwise, we decompose $S$ into separate orbits and prove the result for each orbit independently. Then $S_M=\cup_m S_m$ defines a smooth fiber bundle $S_M\to M$.

The idea is to treat $N$ as the zero set of the function $\Phi\colon S_M\times[0,\tau)\to \R$, $\Phi(s,t)=\la\Omega^{g(t)},s\otimes\bar s\ra_\tr-F(s)$ and to prove certain differential inequality for $\Phi$, which makes it possible to apply Proposition 4 of~\cite{br-sc-08}.

First, note that by assumption, $\Phi(s,0)$ is non-negative on $S_M$, and by Theorem~\ref{thm:main} the same holds for $t>0$. By the evolution equation for $\Omega$,~\eqref{eq:R_evolution_2}, function $\Phi\colon S_M\to \R$ satisfies equation
\begin{equation}\label{eq:br_sc_proof}
\pd_t\Phi(s,t)=\Delta_{h}\Phi(s,t)+\la \Omega^2,s\otimes\bar s\ra_\tr+\la\Omega^\#,s\otimes\bar s\ra_\tr+\sum_{i<j}|\la \n T_{ij},s\ra_\tr|^2+\la\ad_v\Omega,s\otimes\bar s\ra_\tr,
\end{equation}
where $\Delta_h$ is the horizontal Laplacian of the connection $\n^T$. It is defined as follows. Let $\{X_i\}$ be a $g(t)$-orthonormal collection of vector fields in a neighbourhood of $m\in M$. Denote $Y_i:=\n_{X_i}X_i$, and let $\{\widehat{X_i}\}$, $\{\widehat{Y_i}\}$ be the $\n^T$-horizontal lifts of these vector fields to $S_M$. Then $\Delta_h\Phi(s,t):=\sum_i( \widehat{X_i}\!\cdot\!\widehat{X_i}\!\cdot\!\Phi-\widehat{Y_i}\!\cdot\!\Phi)$. 

The second and the fourth summands on the right hand side of~\eqref{eq:br_sc_proof} are always non-negative. Using the same bounds for the remaining terms, as in~\cite[Th.\,A.1]{wi-13}, we conclude, that on a small relatively compact coordinate neighbourhood of $s\in S_M$, for a positive constant $K$ we have
\[
\sum_i \widehat{X_i}\!\cdot\!\widehat{X_i}\!\cdot\!\Phi\le
-K\inf \{\mathrm{Hess}(\Phi)(a,a)\ \bigl|\ |a|\le 1\}+K|\mathrm{grad}(\Phi)|.
\]
At this point, we can apply~\cite[Prop.\,4]{br-sc-08} and conclude that the zero set of $\Phi$ is invariant under the flow generated by the vector fields $\widehat{X_i}$. Since these fields span the $\n^T$-horizontal subspaces on the fiber bundle $S_M\to M$, we obtain that zeros of $\Phi$ are invariant under the $\n^T$-parallel transport.

Now, let us prove (a) and (b). Let $s_0\in S_M$ be a zero of $\Phi(s_0,t_0)$ for some $t_0>0$. Hence function $\Phi(s_0,t_0)$ attains a local minimum at $s=s_0$, $t=t_0$, so $\pd_t \Phi(s_0,t_0)=0$. As we have proved in Theorem~\ref{thm:main}, all summands on right hand side of~\eqref{eq:br_sc_proof} are non-negative, therefore they must vanish. In particular: (a) $\la \Omega^2,s_0\otimes\bar s_0\ra_\tr=0$, so $s_0\in\mathrm{Ker}\,\Omega$ (see definition of $\Omega^2$); (b) $\la\n T_{ij},s_0\ra_\tr=0$.
\end{proof}

Theorem~\ref{thm:br-sc} implies a more familiar version of maximum principle.
\begin{corollary}
	Consider an $\Ad\,G$-invariant subset $S\subset \mf g$ and a nice function $F\colon \mf g\to \R$. Let $(M,g,J)$ be an Hermitian manifold with metric $g=g(t)$, $t\in[0,\tau)$ evolved along the HCF. Assume that
	$\Omega^{g(0)}$ satisfies $C(S,F)$, and that there exists $m_0\in M$ such that $\Omega^{g(0)}$ satisfies strict inequalities, defining $C(\bar S,F)$ and $C(\pd_\infty S,F_\infty)$, at $m_0$:
	\begin{equation}\label{eq:strong_max_strict_inequalities}
	\begin{split}
		\la\Omega_{m_0}^{g(0)},s\otimes\bar s\ra_\tr&> F(s), \mbox{ for any }s\in \bar{S},\\
		\la\Omega_{m_0}^{g(0)},s\otimes\bar s\ra_\tr&> F_\infty(s), \mbox{ for any }s\in \pd_\infty S.
	\end{split}
	\end{equation}
	Then for any $t\in(0,\tau)$, inequalities~\eqref{eq:strong_max_strict_inequalities} hold everywhere on $M$.
\end{corollary}
\begin{proof}
	We claim that $\Omega_{m}$ lies in the interior of $C(S,F)$ if and only if $\Omega_m$ satisfies inequalities~\eqref{eq:strong_max_strict_inequalities}.
	Indeed, if $\Omega$ lies on the boundary of $C(S,F)$, then, following the proof of Theorem~\ref{thm:ode_invariance}, we can find a support functional of the form $\la\cdot,s\otimes s\ra_\tr$, $s\in \bar S\cup\pd_\infty S$, such that in~\eqref{eq:strong_max_strict_inequalities} we have equality. Conversely, if for some $s\in \bar S\cup\pd_\infty S$ we have equality in~\eqref{eq:strong_max_strict_inequalities}, then there is $y\in\Sym^{1,1}(\End(T_m^{1,0}M))$ arbitrary close to $\Omega_{m}$, such that $y\not\in C(S,F)$. So, $\Omega_{m}\in\pd C(S,F)$.
	
	Pick $t_\varepsilon>0$ such that $\Omega_{m_0}^{g(t)}$ still lies in the interior of $C(S,F)$ for $t\in (0,t_\varepsilon)$. Fix any $t\in(0,t_\varepsilon)$.
	By Theorem~\ref{thm:br-sc} the set $N(t,m)$ for $(\bar S,F)$ is invariant under $\n^T$-parallel transport. At the same time, $N(t,m_0)$ is empty, therefore for any $m\in M$ the set $N(t,m)$ is empty as well, i.e., $\la\Omega^{g(t_\varepsilon)},s\otimes\bar s\ra_\tr>F(s)$ for any $s\in\bar S$ everywhere on $M$. Applying the same reasoning to $(\pd_\infty S,F_\infty)$, we conclude that $\la\Omega^{g(t_\varepsilon)},s\otimes\bar s\ra_\tr>F_\infty(s)$ for any $s\in\pd_\infty S$ everywhere on $M$.
	
	Hence, $\Omega^{g(t)}$, lies in the interior of $C(S,F)$. By the proof of Hamilton's maximum principle, $\Omega$ remains in the interior of the convex set $C(S,F)$ for all $t\in(0,\tau)$.
\end{proof}

\subsection{Monotonicity under HCF}
Theorem~\ref{thm:main} allows to produce many monotonic quantities for the HCF on $(M,g,J)$. Let $S\subset \mf g$ be a closed scale-invariant, $\Ad\,G$-invariant subset. Define
\[
\mu(S,g):=\max\{\mu\in\R\ |\ \Omega\in C(S,F),\ F(s)=\mu|\tr s|^2\}\in\R\cup\{\pm\infty\},
\]
where $\max\{\varnothing\}:=-\infty$.
\begin{proposition}\label{prop:monotonicity}
	Let $g=g(t)$ be the solution to the HCF on $(M,g,J)$. Then for any $S\subset \mf g$ as above the quantity $\mu(S,g(t))$ is non-decreasing along the HCF. Moreover, $\mu(S,g(t))>\mu(S,g(0))$ for $t>0$, unless $\mu(S,g(0))\in\{-\infty,0,+\infty\}$.
\end{proposition}
\begin{proof}
	If $\Omega^{g(0)}$ satisfies $C(S,\mu|\tr s|^2)$, then by Theorem~\ref{thm:main} $\Omega^{g(t)}$ also satisfies $C(S,\mu|\tr s|^2)$. Hence $\mu(S,g(t))\ge \mu(S,g(0))$. 
	
	Now assume $\mu(S,g(0))\not\in\{-\infty,0,+\infty\}$, but $\mu(S,g(t))=\mu(S,g(0))=\mu$. Therefore for any $\varepsilon_i>0$ we have $\Omega^{g(t)}\not\in C(S,(\mu-\varepsilon_i)|\tr s|^2)$. Letting $\varepsilon_i\searrow 0$, we conclude that $\Omega^{g(t)}$ hits the boundary of $C(S,\mu|\tr s|^2)$ (here we are using closedness and scale-invariance of $S$). Therefore, in notations of Theorem~\ref{thm:br-sc}, $N(t,m)\neq \varnothing$. By Theorem~\ref{thm:br-sc} (a), it can happen only if $\mu=0$.
\end{proof}
A similar statement is valid for other one-parametric monotonic family of functions $F(s)$.
\medskip

For $S=\{\mathrm{\lambda\,Id\ |\ \lambda\in\R}\}$ Proposition~\ref{prop:monotonicity} gives the monotonicity of the lower bound for $\scal$, see Example~\ref{ex:scalar_curvature}. Of course, this monotonicity can be deduced by applying the usual parabolic maximum principle to the equation
\begin{equation}\label{eq:s_hat_equation}
\pd_t\scal=\Delta\scal+|S^{(3)}|^2+\frac{1}{2}|\mathrm{div}\,T|^2.
\end{equation}
Note that $\inf_M\scal$ is strictly increasing, unless $\scal\equiv 0$, and $S^{(3)}=0$, $\mathrm{div}\,T=0$.

This monotonicity with the vanishing of $S^{(3)}$ and $\mathrm{div}\,T=0$ have important consequences for the understanding of \emph{periodic} and \emph{stationary solutions} to the HCF. In the K\"ahler setting the metrics fixed by the K\"ahler-Ricci flow are tautologically the Ricci flat (Calabi-Yau) metrics. At first glance the situation with the HCF is much more subtle, since vanishing of the term $S^{(2)}+Q$ in~\eqref{eq:HCF} does not have any clear cohomological interpretation. Surprisingly, we still can conclude that $c_1(M)=0$.

More generally, assume that $g=g(t)$ is a \emph{periodic} solution to the HCF~\eqref{eq:HCF} on some time interval $[0,t_{\max})$, i.e., $g(0)=g(T)$ for some $0<T<t_{\max}$.
Hence $\inf_M\scal$ is constant on $[0,T]$, and by the discussion above it follows $\mathrm{div}\,T=0$, $S^{(3)}=0$. We claim that the vanishing of $S^{(3)}$ and $\mathrm{div}\,T$ imply that $c_1(M)=0$ in $H^2(M,\C)$. Indeed, for the first Chern-Ricci form
\[
	\rho:=\sqrt{-1}S^{(1)}_{k\bar s}dz^k\wedge d\bar{z^s},
\]
we have $[\rho]=2\pi c_1(M)$. Now the claim follows from a simple lemma.
\begin{lemma}
	Differential 2-forms $\rho$ and 
	\[
	\rho^T:=
	\frac{\sqrt{-1}}{2}(\mathrm{div}\,T)_{jk}dz^j\wedge dz^k
	+
	\sqrt{-1}S^{(3)}_{k\bar s}dz^k\wedge d\bar{z^s}
	\]
	are cohomologous.
\end{lemma}
\begin{proof}
	Let $\alpha=\sqrt{-1}T_{kp}^p dz^k=\bar{\pd^*}\omega$ be the $(1,0)$-part of the Lee form of $(M,g,J)$. Then, by a standard formula relating the exterior and the covariant derivatives, we get
	\begin{equation}
	\begin{split}
		\bar\pd\alpha&=-\sqrt{-1}\n_{\bar s}T_{kp}^p dz^k\wedge d\bar{z^s}\\
		\pd \alpha&=
		\sqrt{-1}(\n_iT_{kp}^p+\frac{1}{2}T_{ik}^rT_{rp}^p
		)dz^i\wedge dz^k=\frac{\sqrt{-1}}{2}
		(\n_iT_{kp}^p-\n_kT_{ip}^p+T_{ik}^rT_{rp}^p)
		dz^i\wedge dz^k=\\&=
		\frac{\sqrt{-1}}{2}\n_{p}T_{ki}^p dz^i\wedge dz^k=
		-\frac{\sqrt{-1}}{2}(\mathrm{div}\,T)_{ik}dz^i\wedge dz^k,
	\end{split}
	\end{equation}
	where in the last line we used the differential Bianchi identity.
	\[
		\n_i T_{jk}^l+\n_k T_{ij}^l+\n_j T_{ki}^l=
		T_{ij}^pT_{kp}^l+T_{jk}^pT_{ip}^l+T_{ki}^pT_{jp}^l.
	\]
	Using now the first Bianchi identity
	\[
		\Omega_{k\bar s m}^{\ \ \ m}+\n_{\bar s}T_{km}^m=
		\Omega_{m\bar s k}^m,
	\]
	we conclude $\rho-\rho^T=d\alpha$.
\end{proof}

Form $\rho^T$ has a clear geometric interpretation. It is the curvature form of the connection $\n^T$ induced on the anticanonical bundle $-K_M$. A priori $\n^T$ does not preserve any metric, so its curvature form $\rho^T$ might contain a part of type $(2,0)$. Since $\n^T$ is compatible with the holomorphic structure, i.e., $(\n^T)^{0,1}=\bar\pd$, we can use $\n^T$-parallel transport to construct a nowhere vanishing holomorphic section of $\pi^*K_M$, where $\pi\colon\widetilde{M}\to M$ is the universal cover. This proves the following result.
\begin{theorem}\label{thm:trivial_canonical_bundle}
	If a compact complex manifold $M$ admits an HCF-periodic Hermitian metric, then the pull-back of the canonical bundle to the universal cover of $M$ is holomorphically trivial.
\end{theorem}
The statement of the above theorem is stronger then just vanishing of $c_1(M)\in H^2(M,\C)$. For example, Calabi-Eckman complex structures on $S^3\times S^3$ have holomorphically non-trivial canonical bundle, while $c_1=0$, see, e.g., discussion in~\cite[\S 2]{ba-do-ve-09}.

It is still an open question, whether the HCF~\eqref{eq:HCF} admits non-trivial, i.e., not stationary, periodic solutions.
\begin{problem}
	Is is true that if $g=g(t)$ is a periodic solution to the HCF on $(M,g,J)$, then $g(t)$ is a stationary solution, i.e., $g\equiv g(0)$?
\end{problem}

Theorem~\ref{thm:trivial_canonical_bundle} motivates us to formulate the following problem.
\begin{problem}
	Let $(M, J)$ be a compact complex manifold with a trivial canonical bundle. Does $M$ admit an HCF-stationary metric? By Theorem~\ref{thm:trivial_canonical_bundle}, such a metric necessarily will have $S^{(3)}=0$, $\mathrm{div}\,T=0$.
\end{problem}
This problem is a non-K\"ahler version of Calabi's conjecture. Namely, if the underlying manifold $(M,J)$ admits a K\"ahler metric $\omega$, then by Yau's theorem, there exists a unique K\"ahler metric $\omega_\phi=\omega+i\pd\bar{\pd}\phi$, such that $Ric(\omega_\phi)=0$. Of course, in this case the torsion vanishes, and all four Chern-Ricci curvatures coincide and equal zero.
\bibliographystyle{abbrv}
\bibliography{biblio}
\end{document}